\documentclass[11pt]{amsart} 

\usepackage{amsthm, amssymb, enumerate, mathtools, subcaption, amsfonts, amssymb, tikz, verbatim, mathptmx, floatrow, float, hyperref, float,listings, graphicx, todonotes, textcomp, mathtools, abraces}
\usepackage[margin=3cm]{geometry}
\usepackage[rightcaption]{sidecap}
\usepackage{hyperref}
\usepackage[nameinlink, capitalize]{cleveref}
\usepackage[normalem]{ulem}
\usepackage{pgfplots}
\pgfplotsset{compat=1.12,}
\graphicspath{ {./img/} }

\DeclarePairedDelimiter\ceil{\lceil}{\rceil}
\DeclarePairedDelimiter\floor{\lfloor}{\rfloor}

\usepackage[compatibility=false]{caption}
\makeatletter
\let\x@caption\caption 
\long\def\x@@caption[#1]#2{\x@caption[{#1}]{#1 --- #2}}
\def\x@@@caption#1{\x@caption[{#1}]{#1}}
\AtBeginDocument{%
  \def\caption{\@ifnextchar[\x@@caption\x@@@caption}
}
\makeatother

\newtheoremstyle{astyle}
    {12pt}       
    {}             
    {}            
    {}                       
    {\scshape \bf}             
    {.}                        
    {.5em}                    
    {}  

\theoremstyle{astyle}\newtheorem{definition}{Definition}
\theoremstyle{astyle}
\theoremstyle{astyle}\newtheorem{remark}{Remark}
\theoremstyle{astyle}\newtheorem{theorem}{Theorem}
\theoremstyle{astyle}
\theoremstyle{astyle}
\theoremstyle{astyle}\newtheorem{lemma}{Lemma}
\theoremstyle{astyle}\newtheorem{corollary}{Corollary}
\crefname{conj}{Conjecture}{conjectures}

\newcommand{\N}{\mathbb{N}}
\newcommand{\Z}{\mathbb{Z}}
\newcommand{\R}{\mathbb{R}}
\newcommand{\Q}{\mathbb{Q}}

\newcommand{\mycomment}[1]{}

\definecolor{codegreen}{rgb}{0,0.6,0}
\definecolor{codegray}{rgb}{0.5,0.5,0.5}
\definecolor{codepurple}{rgb}{0.58,0,0.82}
\definecolor{backcolour}{rgb}{0.95,0.95,0.92}

\lstdefinestyle{mystyle}{
    backgroundcolor=\color{backcolour},   
    commentstyle=\color{codegreen},
    keywordstyle=\color{magenta},
    numberstyle=\tiny\color{codegray},
    stringstyle=\color{codepurple},
    basicstyle=\ttfamily\footnotesize,
    breakatwhitespace=false,         
    breaklines=true,                 
    captionpos=b,                    
    keepspaces=true,                 
    numbers=left,                    
    numbersep=5pt,                  
    showspaces=false,                
    showstringspaces=false,
    showtabs=false,                  
    tabsize=2
}

\makeatletter
\newcommand{\addresseshere}{%
  \enddoc@text\let\enddoc@text\relax
}
\makeatother
  
\title{Properties of $k$-Descending Trees}
\author{Agniv Sarkar \and Eric Severson}
\date{2021-2022}

\email{\href{mailto:agnivsarkar@proofschool.org}{agnivsarkar@proofschool.org} \and \href{mailto:eseverson@proofschool.org}{eseverson@proofschool.org}}
\address{973 Mission St, San Francisco, CA}

\begin{document}


\begin{abstract}
This was research presented at the Worldwide Federation of National Math Competitions in Bulgaria in 2022. 

For any real-valued $k > 1$, we consider the tree rooted at 0, where each positive integer $n$ has parent $\lfloor\frac{n}{k}\rfloor$. The average number of children per node is $k$, thus this definition gives a natural way to extend $k$-ary trees to irrational $k$. We focus on the sequence $r_d$: the count of nodes at depth $d$.

We first prove there exists some constant $\rho(k)$ such that $r_d \sim \rho(k)\cdot k^d$. We then study a family of values $k=\frac{a + \sqrt{a+4b}}{2}$, where we prove the sequence satisfies the exact recurrence $r_d = a\cdot r_{d-1} + b \cdot r_{d-2}$. This generalizes a special case when $k$ is the golden ratio and $r_d$ is the Fibonacci sequence.
\end{abstract}
\maketitle

\tableofcontents

\section{Introduction}
\label{sec:intro}

This contains a piece of original research centered within graph and number theory with a surprising connection to the Josephus Problem. The research was carried out by myself (Agniv Sarkar) and my mentor (Eric Severson) throughout the high school year of 2021-2022. 

This research was done to observe patterns seen in the $\phi$-tree generated with \cref{def:k-tree}. This became a very nice number theoretic problem, and when we began to look at the asymptotics of these trees, we found that there was a connection to the Josephus problem and calculating the solution to the problem. 

The trees themselves are most similar to a $k$-ary tree \cref{def:k-ary-tree}.

\begin{definition}
\label{def:k-ary-tree}
A $k$-ary tree is a rooted tree such that each node has no more than $k$ children. A complete $k$-ary tree is a rooted tree such that each node has exactly $k$ children. 
\end{definition}

This tree is commonly used as a data structure in computer science, such as through a Binary Search Tree, or a $2$-ary tree. However, the $k$ in the definition does not generalize nicely to non integer $k$, and that is where \cite{marsault2014rhythmic} defines a ``rhythmic tree," which is \cref{def:rhythm-tree}.

\begin{definition}
\label{def:rhythm}
Let $p, q \in \mathbb{Z}$ such that $p > q \geq 1.$ Then, 
\begin{itemize}
  \item Rhythm of directing parameter $(q, p)$ is a $q$-tuple of $r$ non-negative integers whose sum is $p$. 
  \[r = (r_0, r_1, \ldots, r_{q-1}) \text{, and, } \sum_{i=0}^{q-1}r_i = p.\]
  \item A rhythm $r$ is valid if it also satisfies 
  \[\forall k \in \{0, 1, \ldots, q-1\}, \sum_{i = 0}^j r_i > j + 1\]
  \item The growth rate of $r$ is the rational number $\frac{p}{q}$ or $\frac{p'}{q'}$ where $p'$ and $q'$ are quotients of $p$ and $q$ by their greatest common divisor, such that they are coprime.
\end{itemize}
\end{definition}

\begin{definition}
\label{def:rhythm-tree}
Let $r = (r_0, \ldots, r_{q-1})$ be a valid rhythm as given by \cref{def:rhythm}. Then the rhythmic tree $\mathcal{I}_r$ generated by $r$ is defined by:
\begin{itemize}
    \item the root $0$ of $\mathcal{I}_r$ has $(r_0 - 1)$ children, which are the notes $1, 2, \ldots,$ and $(r_0 - 1)$. 
    \item for $n > 0$, the node $n$ has $r_{n \mod q}$ children, which are the nodes $(m+1), (m+2), \ldots, $ and $(m + r_{n \mod q})$ where $m$ is the largest child of $(n-1)$.
\end{itemize}
\end{definition}

In \cite{marsault2014rhythmic}, they prove that if $r$ has rational but not integer growth rate, then the paths to all vertices in the rhythmic tree $\mathcal{I}_r$ cannot be verified with a finite automaton. If $r$ is a rhythm with integer growth rate, then paths in the tree $\mathcal{I}_r$ can be described with a finite automaton. 

The reason that rhythmic trees are relevant and the previous statement about their structure is that our new definition is a subset of all rhythmic trees when $k$ is a rational number.

Also, \cite{odlyzko1991functional} describes the Josephus problem, where we are given two numbers, $n, q$. There are $n$ places arranged in a circle, and each $q$th person is excused from the circle until only $1$ is left. The problem is to find out what index the person who is left remains, which is simple for $q = 2$, harder for $q = 3$, and there exists a recursive solution for a general $n, q$. This paper defines the function $c(k).$

For some $k > 1,$ fix the recursive sequence $f_n$ to be $f_0 = 1, f_{n+1} = \ceil{kf_{n}}.$ Then, the function $c(k)$ is defined to be the constant $c$ such that $f_{n} \sim c(k) k^n$ as $n$ trends to infinity. We expand and prove more results about this function in \cref{sec:asy}.

\section{Preliminaries}
\label{sec:defs}

\subsection{Notation}

$\N$ denotes the set of non-negative integers, $\Q$ the set of rational numbers, and $\R$ the set of real numbers. $\floor{\cdot}$ and $\ceil{\cdot}$ denote the floor and ceiling functions. For $x\in\R$, $\{x\}:=x-\floor{x}$ denotes the fractional part of $x$. Note that this means that $\{x\} : \R \longrightarrow [0, 1)$. $\sim$ denotes asymptotic equivalence, where $f(n) \sim g(n)$ means $\lim_{n\to\infty}\frac{f(n)}{g(n)} = 1$.

\subsection{Definitions}

We first formally define a $k$-descending tree:

\begin{definition}
\label{def:k-tree}
For any $k \in (1,\infty) \subset \R$, the \emph{$k$-descending tree}, or simply \emph{$k$-tree} is the rooted tree with nodes in $\N$, where every $n \in \N$ has the parent $\floor{\frac{n}{k}}$, and $0$ is the root node.
\end{definition}

\cref{fig:k_3}, \cref{fig:k_3/2}, \cref{fig:k_phi} were generated with \cref{appendix:tree-code}.

For the integer $k$ case, it generates a complete $k$-ary tree as shown in \cref{fig:k_3}. 

\begin{figure}[H]
    \centering
    \includegraphics[width=\textwidth]{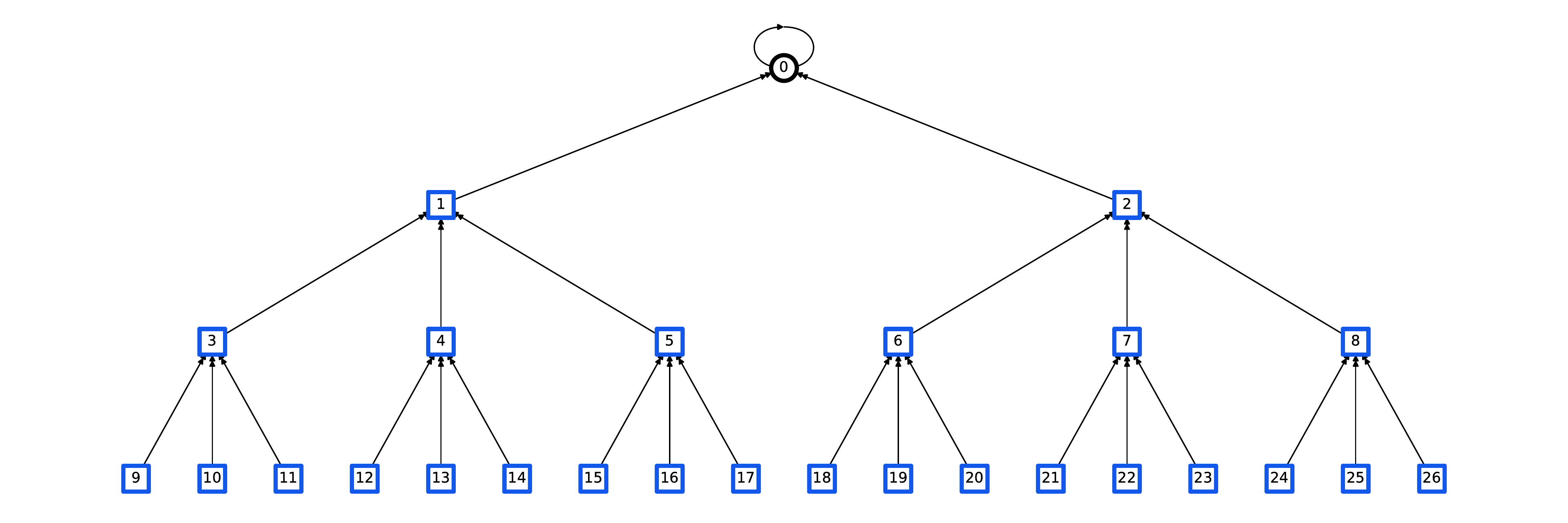}
    \caption[$3$-descending tree]{Used code to generate this tree.}
    \label{fig:k_3}
\end{figure}

Note that in \cref{fig:k_3} that it is the same as a rhythmic tree with rhythm $(3).$ This is true for all integers.

Then, for rational $k$, it generates another rhythmic tree as described in \cite{marsault2014rhythmic}. The childcounts in \cref{fig:k_3/2} are periodic with pattern $2, 1, \ldots$, meaning that it is a rhythmic tree with rhythm $(2, 1).$

\begin{figure}[H]
    \centering
    \includegraphics[width=\textwidth]{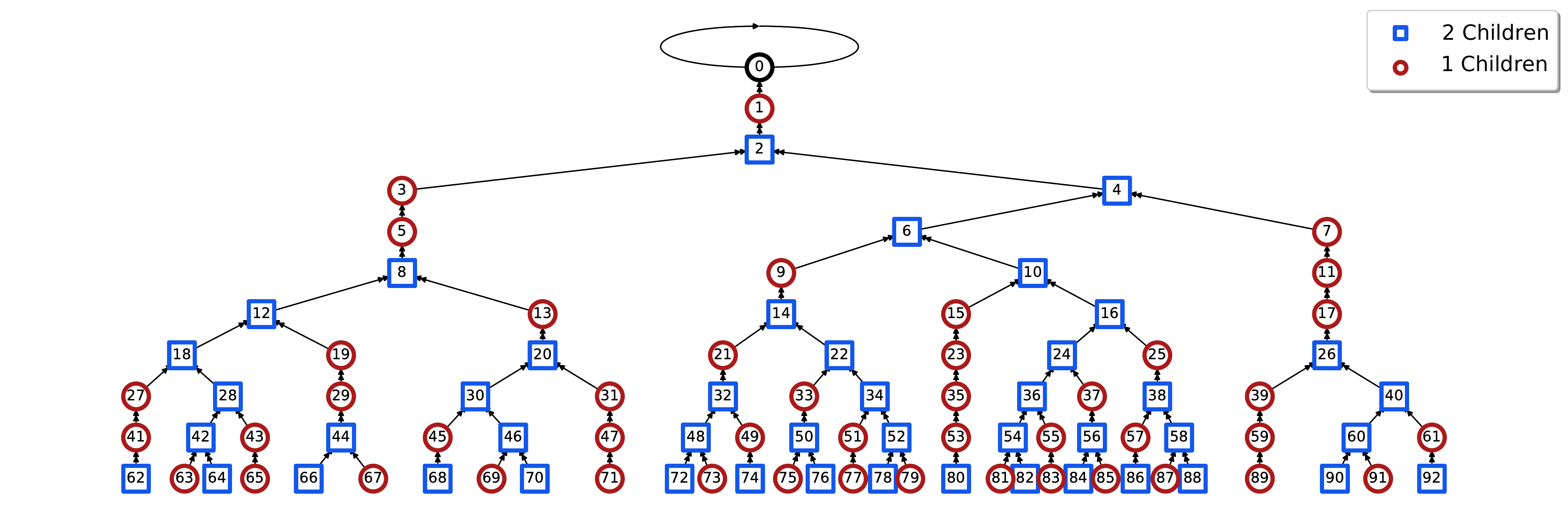}
    \caption[$\dfrac{3}{2}$-descending tree]{Coded it to generate this tree.}
    \label{fig:k_3/2}
\end{figure}

Then, when $k$ is irrational, the behavior becomes slightly more chaotic. This specific example in \cref{fig:k_phi} shows some particularly nice behavior due to our choice of $k = \phi$. 

\begin{figure}[H]
    \centering
    \includegraphics[width=\textwidth]{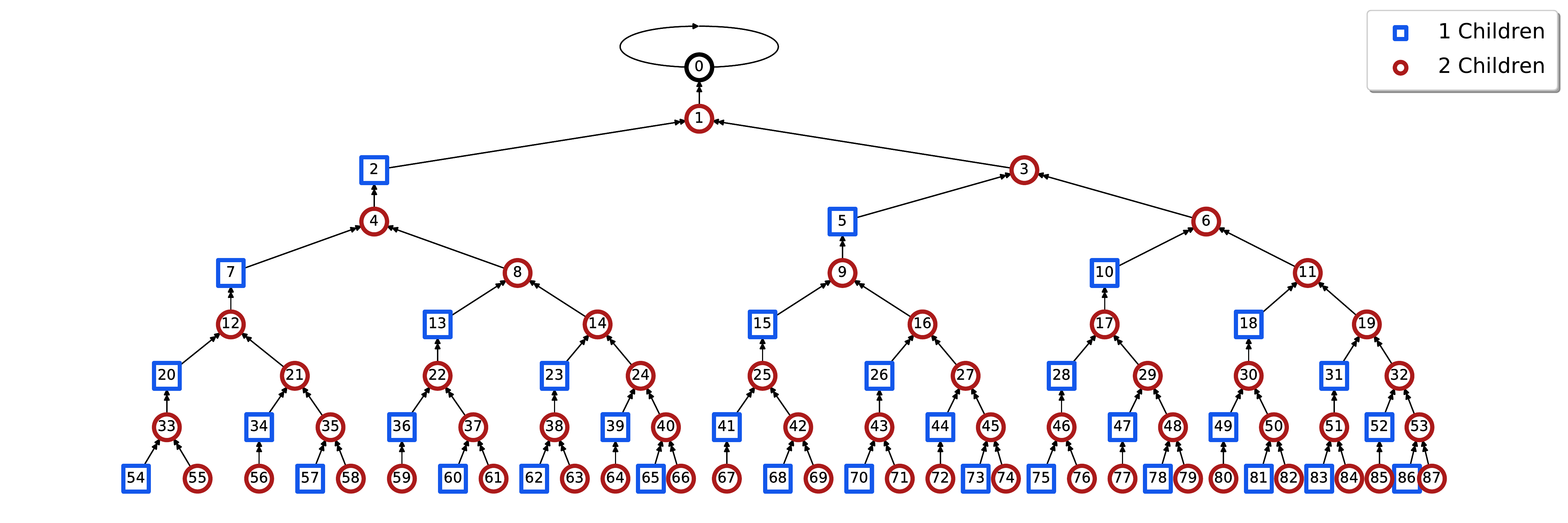}
    \caption[$\frac{1+\sqrt{5}}{2}\approx 1.618$-descending tree]{Used a decimal approximation (up to 1000 digits) to generate this tree.}
    \label{fig:k_phi}
\end{figure}

Note that this does not have a clearly obvious rhythm, with the childcounts not appearing to be periodic.

The remaining functions are defined based on a particular $k$-tree, so are technically also functions of $k$, but for brevity we will often not write the dependence on $k$ explicitly.

\begin{definition}
\label{def:child-count}
For any $n\in\N$, let $\text{children}(n) = \{c \in \N : \floor{\frac{c}{k}} = n\}$ be the set of children of $n$ in the $k$-tree\footnote{Note that technically $0\in\text{children}(0)$. This makes $h(0) = \ceil{k}$, which is consistent with \cref{lem:child-count}, as well as \cref{def:rhythm-tree}.}. Then $h(n) = |\text{children}(n)|$ gives the \emph{child-count} of $n$.
\end{definition}

\begin{remark}
\label{rem:smallest-child}
For any $n\in\N_+$, $\min(\text{children}(n)) = \ceil{n\cdot k}$.
\end{remark}

\begin{proof}
This follows immediately from the fact that $\floor{\frac{\ceil{nk}}{k}} = n$, while $\floor{\frac{\ceil{nk}-1}{k}} < n$.
\end{proof}

We can now formally prove the claim that the average child-count is $k$:

\begin{remark}
\label{rem:average-children}
$\lim_{N\to\infty} \frac{1}{N}\sum_{n=0}^{N-1} h(n) = k$.
\end{remark}

\begin{proof}
Notice that the sum $\sum_{n=0}^{N-1} h(n)$ counts every node from $0$ to the largest child of the node $N-1$. Thus,
\[\sum_{n=0}^{N-1} h(n) = 1 + \max(\text{children}(N-1)) = \min(\text{children}(N)) = \ceil{N\cdot k},\]
from \cref{rem:smallest-child}. Then $\lim_{N\to\infty}\frac{1}{N} \ceil{N\cdot k} = k$.
\end{proof}

The child count $h(n) \in \{\floor{k}, \ceil{k}\}$, and turns out to depend only on the quantity $\{n\cdot k\}$:

\begin{definition}
\label{def:child-count-ind}
For $n\in \N_+$, we call the fractional part $\{n\cdot k\} \in [0,1)$ the \emph{count indicator}. The interval $(0,1-\{k\}] \subset [0,1)$ is called the \emph{floor-range} and its complement $\{0\} \cup (1-\{k\}, 1)$ is called the \emph{ceil-range}.
\end{definition}

These definitions are motivated by the following foundational lemma:

\begin{lemma}
\label{lem:child-count}
For all $n\in\N$,
$$h(n) = \begin{cases}
  \floor{k}  & \emph{\text{if }} \{n \cdot k\} \in (0, 1 - \{k\}]\\
  \ceil{k} &  \emph{\text{otherwise}}.
\end{cases}$$
\end{lemma}

\begin{proof}
This statement is trivial if $k\in\Z$, so we assume now $k \notin \Z$ and $\floor{k} + 1 = \ceil{k}$.

Notice that $\floor{\frac{c}{k}} = n$ iff $c \in [nk, (n+1)k)$, thus $h(n)$ counts the integer points in this interval, whose width is $k = \floor{k} + \{k\}$. Intuitively, the first part of the interval $[nk, nk + \floor{k})$ of width $\floor{k}$ will always contain $\floor{k}$ integer points. Then when $\{nk\} > 1-\{k\}$, the remaining interval $[nk + \floor{k}, nk + k)$ will ``wrap around'' one additional integer point. On the other hand, when $\{nk\} < 1-\{k\}$, the fractional part strictly increases and does not wrap around an extra integer point.

In the boundary case $\{nk\} = 0$, $nk \in \Z$, and the interval contains $\ceil{k}$ integer points $nk, nk+1, \ldots, nk + \floor{k}$.

In the other boundary case $\{nk\} = 1 - \{k\}$, we have $nk+k \in \Z$, but the interval does not contain this rightmost boundary, and we thus have $\floor{k}$ integer points $nk+k-1, nk+k-2, \ldots, nk+k-\floor{k}$ in the interval.
\end{proof}

Notice that \cref{lem:child-count} implies that for rational $k\in \Q$, the child-count function $h(n)$ is periodic. See for example \cref{fig:k_3/2}. Due to this periodicity we see that it is a rhythmic tree \cref{def:rhythm-tree}.

Finally, we formally define the row-length sequence:

\begin{definition}
\label{def:row-lengths}
$(r_d)_{d=0}^\infty$ is the \emph{row-length sequence}, where $r_d$ gives the number of nodes at depth $d$ in a $k$-tree. More formally, for a node $n\in\N$, the iterated function sequence $g_0=n$ and $g_{i+1} = \floor{\frac{g_i}{k}}$ gives the path to the root. Then $\text{depth}(n) = \max(i:g_i = 0)$ and $r_d = |\{n\in\N : \text{depth}(n) = d\}|$.
\end{definition}

The sequence $r_d$ then intuitively grows at an exponential rate of $k$ as how $r_d \approx r_{d-1} \cdot k.$ However, due to the rounding off done by the floor function, it is not exactly this. So, we can define an asymptotics function on the row lengths.

\begin{definition}
\label{def:rho}
$\rho(k) = \lim_{d\to\infty}\frac{r_d}{k^d}$.
\end{definition}

Note that we still need to show this limit exists. 

\section{Asymptotics of \texorpdfstring{$r_d$}{\textit{rd}}}
\label{sec:asy}

\begin{theorem}
\label{thm:rho_exists}
For any $k \in (1,\infty) \subset \R$, the constant $\rho(k) = \lim_{d\to\infty}\frac{r_d}{k^d}$ exists. Thus $r_d \sim \rho(k) \cdot k^d$.
\end{theorem}

\begin{proof}
This is essentially a corollary of Proposition 1 from \cite{odlyzko1991functional}. To be self-contained, we produce the proof in its entirety.

Observe that for any $n\in\N_+$, $\min(\text{children}(n)) = \ceil{n\cdot k}$. We will then consider the sequence $f_0 = 1$ and $f_{i+1} = \ceil{f_i\cdot k}$. Notice that, subject to a change in indexing, this gives the leftmost elements in each row. For example, see

\begin{table}[H]
    \centering
    \begin{tabular}{ |c|c| } 
 \hline
 $f_i$ & Value \\ 
 \hline
 $f_0$ & 1 \\ 
 $f_1$ & 2 \\ 
 $f_2$ & 4 \\ 
 $f_3$ & 7 \\ 
 $f_4$ & 12 \\ 
 $f_5$ & 20 \\
 \hline
    \end{tabular}
    \caption{$f_i$ values for the $\phi$-tree, shown in \cref{fig:k_phi}.}
    \label{fig:f_i-example}
\end{table}

We thus have $r_d = f_{d} - f_{d-1}$ for all $d>0$. Proposition 1 \cite{odlyzko1991functional} shows there exists a constant\footnote{\cite{odlyzko1991functional} uses $\alpha$ as the parameter instead of $k$.} $c(k)$ such that $f_i \sim c(k) \cdot k^i$. This will then imply $$r_d \sim c(k)\cdot k^{d} - c(k)\cdot k^{d-1} = \frac{k-1}{k}c(k) \cdot k^d,$$
thus we have $\rho(k) = \frac{k-1}{k}\cdot c(k)$.

To prove Proposition 1 from \cite{odlyzko1991functional}, or that $c(k)$ exists, we consider the sequence $(\frac{f_i}{k^i})_{i=0}^\infty$. First we show the sequence is nondecreasing, since
$$\frac{f_{i+1}}{k^{i+1}} = \frac{\ceil{f_i\cdot k}}{k^{i+1}} \geq \frac{f_i}{k^i}.$$
The sequence is also bounded above. We start with
$$\frac{f_{i+1}}{k^{i+1}} = \frac{\ceil{f_i\cdot k}}{k^{i+1}} < \frac{f_i\cdot k+1}{k^{i+1}} = \frac{f_i}{k^i} + \frac{1}{k^{i+1}},$$
and then with the base case $\frac{f_0}{k^0} = 1$ we conclude
$$\frac{f_n}{k_n} < \sum_{i=0}^n \frac{1}{k^i} <  \sum_{i=0}^\infty \frac{1}{k^i}= \frac{1}{1-1/k} = \frac{k}{k-1}.$$
Since the sequence $(\frac{f_i}{k^i})_{i=0}^\infty$ is nondecreasing and bounded from above, the limit $c(k) = \lim_{i\to\infty}(\frac{f_i}{k^i})$ exists. Moreover, we have the following bounds
$$1 \leq c(k) \leq \frac{k}{k-1}.$$

This proves Proposition 1 from \cite{odlyzko1991functional}, in turn proving the existence of $\rho(k)$ for the asymptotics of $r_d$. 
\end{proof}

\begin{corollary}
\label{cor:rho_bounds}
$\frac{k-1}{k} \leq \rho(k) \leq 1$.
\end{corollary}

\begin{proof}
The bounds on $c(k)$ as described in the proof of \cref{thm:rho_exists} alongside the relationship $\rho(k) = \frac{k-1}{k}\cdot c(k)$ gives the following bounds on $\rho$.
\end{proof}

So, we can now look at approximations of this data. The code for generating these approximations are attached in \cref{appendix:rho-code}. Some of the raw data is contained within \cref{appendix:raw}.

Within \cite{odlyzko1991functional}, Prop 3. they prove that there are jump discontinuities at the ``Josephus points," or rational numbers of the form $\frac{q}{q-1}$ where $q \in \mathbb{Z}_+.$  This is shown within \cref{fig:joseyjumps}. Since this is not relevant to the focus of the paper the proof of this will not be reproduced in its entirety.

\begin{figure}[H]
    \centering
    \includegraphics[width=\textwidth]{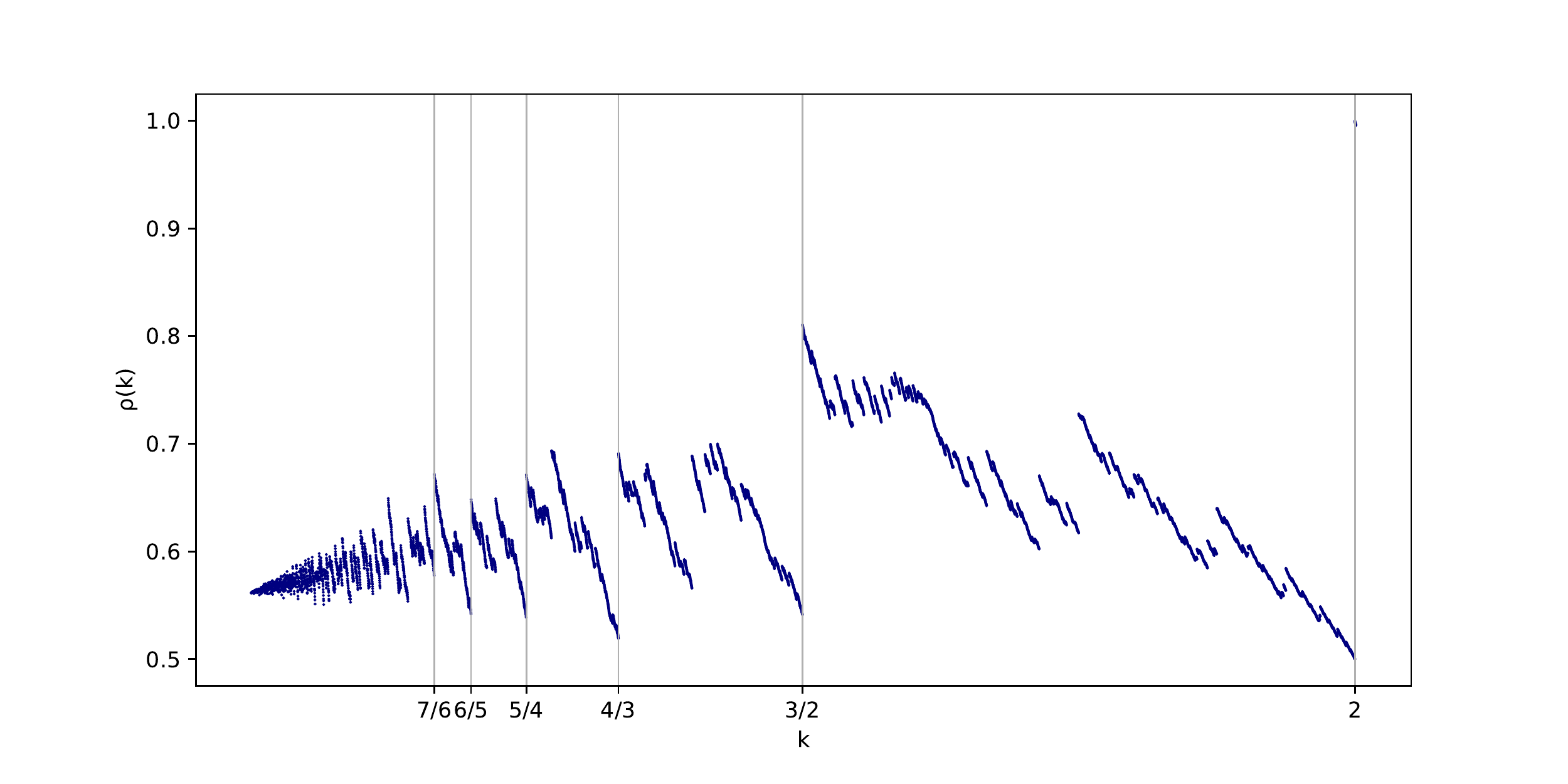}
    \caption[Jump Discontinuities from \cite{odlyzko1991functional}]{In Prop 3. from \cite{odlyzko1991functional}, there are jump discontinuities of an extremely nice form, where $c(a+\epsilon)=ac(a).$}
    \label{fig:joseyjumps}
\end{figure}

Zooming out with \cref{fig:rho_graph}, we can see what appears to be a more global pattern. 

\begin{figure}[H]
    \centering
    \includegraphics[width=\textwidth]{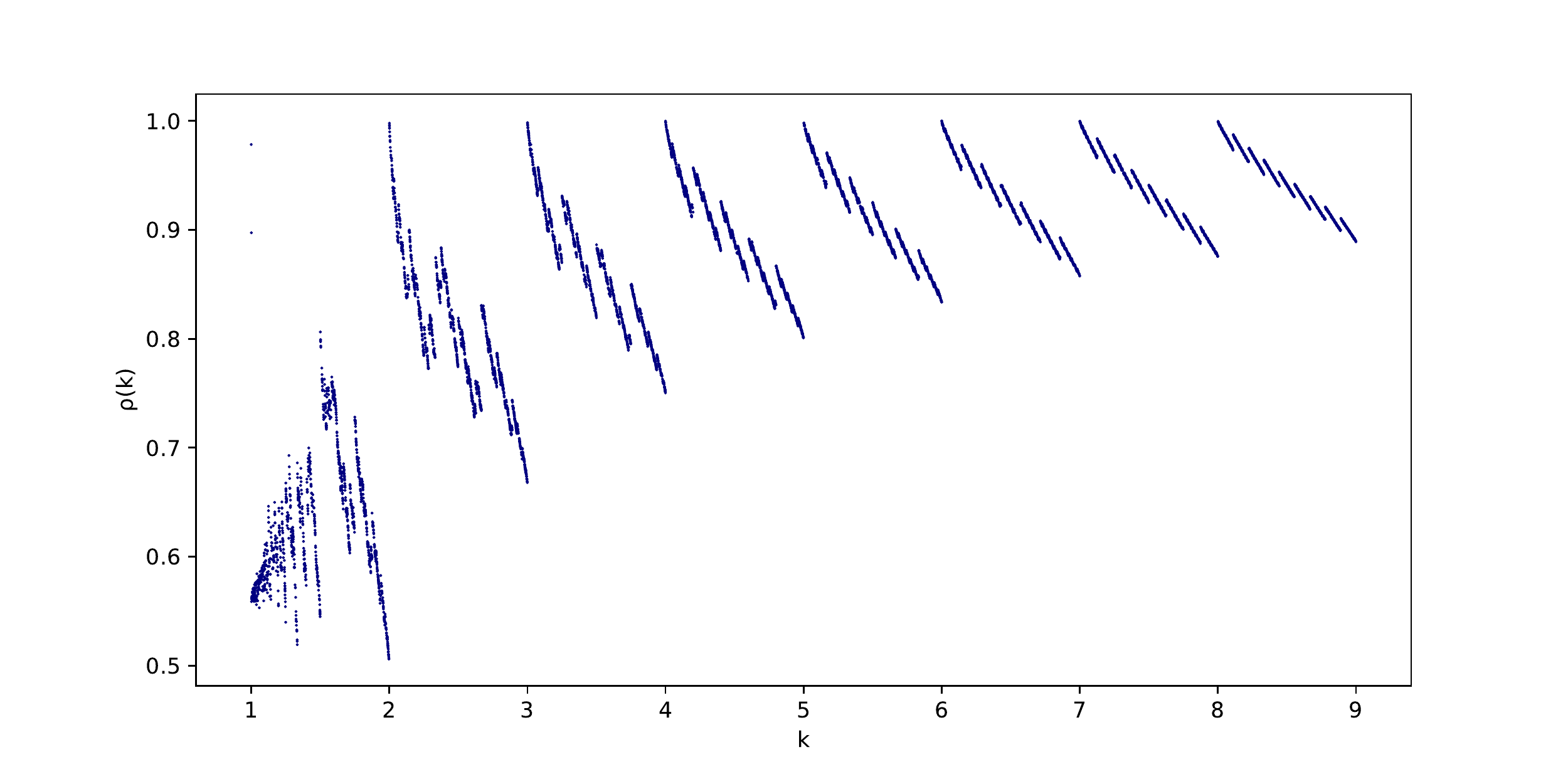}
    \caption[$\rho(k)$ for $1 < k \leq 9$]{Approximated $\rho(k)$ for $10^4$ points.}
    \label{fig:rho_graph}
\end{figure}

One of these patterns appears to be periodic splits within the function. This is illustrated within \cref{fig:splits_rho}. Within $n$ and $(n+1)$ for some integer $n \geq 1$, there are $(n+1)$ periodic `visible splits.' Upon closer inspection, it appears there are also $(n+1)^2$ periodic `visible splits,' with the best ones happening closest to $(n+1).$ This pattern visually holds, and the splits get smaller. Upon numerical testing, it is not clear if they share the same form as those described in \cite{odlyzko1991functional}.

\begin{figure}[H]
    \centering
    \includegraphics[width=\textwidth]{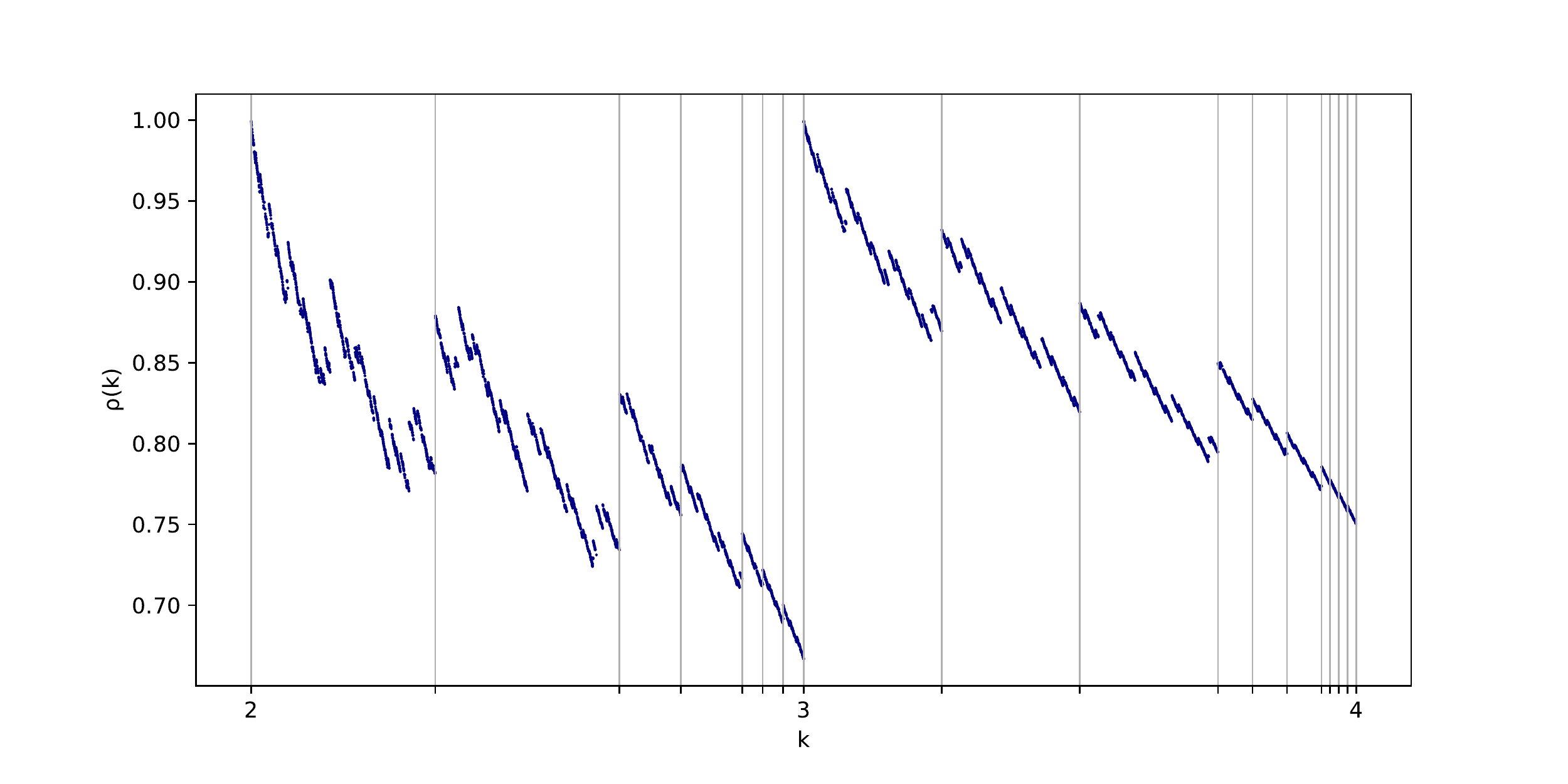}
    \caption[$\rho(k)$ for $1 < k \leq 4$ with apparent splits]{Approximated $\rho(k)$ for $10^4$ points, and also plotted the points of possible visible splits that seem to be self similar between $n$ and $n+1$.}
    \label{fig:splits_rho}
\end{figure}

\section{\texorpdfstring{$k$}{\textit{k}}-Trees with Closed Form \texorpdfstring{$\rho(k)$}{\textit{rho(k)}}}

Let us then define something that makes it more clear how to describe indicators.

\begin{definition}
\label{def:cci-graph}
The child-count indicator graph for some $k > 1$ is the set of functions $\{f_i | i = 1, 2, \ldots, \ceil{k}\}$, where $f_i(x)$ is defined to be the fractional part of the $i$th child of some $n$ where $x = \{nk\},$ so that $f_i : [0, 1) \longrightarrow [0, 1)$. If $\ceil{k} \neq \floor{k}$, then we define $f_{\ceil{k}} : (1-\{k\}, 1] \longrightarrow [0, 1)$ for the $ceil{k}$th child of a node $n$ if it exists. 

Also, the ceil-range and floor-range are used to say $(1-\{k\}, 1]$ and $(0, 1-\{k\}]$ respectively and are helpful to describe the child-count indicator graph in a more general sense. 
\end{definition}

What \cref{def:cci-graph} allows us to do is visualize the behavior of the ``grand-children" of a node $n$. If we know it has fractional part $\{nk\},$ then we can see the fractional parts of its children, showing us its children's children, leading to the terminology of ``grandchildren." 

\begin{lemma}
\label{lem:a-b}
Let $k$ satisfy the equation $k = a + \frac{b}{k}$ for $a,b \in \Z$. Then for any node $n\in \N_+$, let $x = \{n \cdot k\} \in [0,1)$ be the count-indicator for $n$. Let $c_1, \ldots, c_{h(n)}$ be the children of the node $n$. Then the $i$th smallest child $c_i$ has count-indicator
\[\{c_i \cdot k\} = \{(i-\{nk\})\cdot \frac{b}{k}\}.\]
\end{lemma}

\begin{proof}
From \cref{rem:smallest-child}, we have $c_1 = \ceil{n\cdot k}$, and more generally $c_i = \ceil{n\cdot k} + i - 1$. We then have count indicator
\begin{align*}
    \{c_i \cdot k\} &= \{(\ceil{n\cdot k} + i - 1)(a + \frac{b}{k})\} \\
    &= \{(\ceil{n\cdot k} + i - 1)\cdot\frac{b}{k}\} \\
    &= \{(n\cdot k + 1 - \{n\cdot k\} + i - 1)\cdot\frac{b}{k})\} \\
    &= \{(i-\{nk\})\cdot \frac{b}{k}\}.
\end{align*}

\end{proof}

Some examples of \cref{def:cci-graph} and \cref{lem:a-b} are shown in \cref{fig:phi-grand}, \cref{fig:3-1-grand}, \cref{fig:37-grand}, \cref{fig:5-3-grand}, \cref{fig:53-grand}, located in the \cref{appendex:grand-figures}.

\begin{theorem}
\label{thm:grand}
(Grandparent theorem) Let $a, b \in Z$ such that $a \geq 1$ and $1-a \leq b \leq a-1$. When $b \geq 0$, then there are always $b$ distinct lines in the ceil-range of the child count indicator graph for $k = \frac{a+\sqrt{a^2+4b}}{2}$. When $b$ is negative, then there are always $|b|$ distinct lines in the floor-range of the graph of $k$.
\end{theorem}

\begin{proof}
Let $a, b \in Z$ such that $a \geq 1$ and $1-a \leq b \leq a-1$. Then, let $k = \frac{a+\sqrt{a^2+4b}}{2}$. Using \cref{lem:a-b}, \begin{align*}
    \{c_{i+1} \cdot k\} - \{c_{i} \cdot k\} & = \{(i + 1 - x) \cdot \frac{b}{k}\} - \{(i - x) \cdot \frac{b}{k}\} \\ & = \{\frac{b}{k}\} \\ & = \{k\}, 1-\{k\} (\text{depending on } b). 
\end{align*}

So, the difference between two consecutive childcount indicators is $\{k\}$ or $1-\{k\}$. The output ceil-range from \cref{lem:child-count} is of size $\{k\},$ and the floor-range is then $1 - \{k\}.$ So, when one childcount indicator falls above or under that line, the next indicator ``flips,'' meaning that the amount of lines in the output ceil-range and floor-range stays constant throughout $x \in [0, 1).$ So, this simplifies the proof in allowing us to choose for $x$ to show that there are specifically $b$ indicators in the specific output range. 

The last step is now to find the number of solutions for $i$ in this expression, $1-\{k\} < \{\{k\}(i-x)\} < 1,$ but now we can choose $x$. If $x = 0,$ then the set of solutions is $\{\lfloor \frac{a}{b}\rfloor, \lfloor \frac{2a}{b}\rfloor, \lfloor \frac{3a}{b}\rfloor, \ldots, \lfloor \frac{(b-1)a}{b}\rfloor, \lfloor a \rfloor\}$, which contains $b$ solutions (in the nonnegative $b$ case). This is the set of solutions to $0 < \{\{k\}(i-x)\} < 1-\{k\}$ for the negative $b$ case. This then proves the theorem.
\end{proof}

\begin{theorem}
\label{thm:recurrence}
Let $a,b \in \Z$ with $a \geq 1$ and $1-a < b < 1+a$. For $k = \frac{a + \sqrt{a^2 + 4b}}{2}$, the row-length sequence for the $k$-tree satisfies the linear recurrence
\[r_d = a\cdot r_{d-1} + b \cdot r_{d-2},\]
with base case $r_0 = 1$, $r_1 = \ceil{k}-1$.
\end{theorem}

\begin{proof}
Let $a,b \in \Z$ with $a \geq 1$ and $1-a < b < 1+a$. When $b \geq 0,$ we can see that $\floor{k} = a.$ So, $r_d \geq ar_{d-1}$, as from \cref{lem:child-count} we know that each element from the previous row contributes at least $\floor{k}$ children. With \cref{thm:grand}, we can see that each element in $r_{d-2}$ has $b$ children that have $\ceil{k}$ children, which creates the equality $r_d = ar_{d-1} + br_{d-2}.$ Instead, when $b < 0,$ then $\ceil{k} = a,$ so $r_d \leq ar_{d-1}$. What happens then is $br_{d-2}$ actually takes away to compensate for the values with $\floor{k}$ children. So, \[r_d = ar_{d-1} + br_{d-2}.\]
\end{proof}

\begin{corollary}
\label{cor:exact-rho}
Let $a,b \in \Z$ with $a \geq 1$ and $1-a < b < 1+a$. For $k = \frac{a + \sqrt{a^2 + 4b}}{2}$, 
$$\rho(k) = \begin{cases} \dfrac{k}{\sqrt{a^2+4b}} & b > 0 \\[2em] \dfrac{k-1}{\sqrt{a^2+4b}} & b < 0 \\[2em] \dfrac{a-1}{a} & b = 0 \end{cases}$$
\end{corollary}

\begin{proof}
Because of \cref{thm:recurrence}, we can find a closed formula for $r_d$. 

As the theorem states, let $a, b \in \Z$ such that $a \geq 1, 1-a < b < 1+a$, and let $k_1 = \frac{a + \sqrt{a^4 + 4b}}{2}$ and $k_2 = \frac{a - \sqrt{a^4 + 4b}}{2}$.

Then, for the $k_1$ tree, \[r_d = ar_{d-1} + br_{d-2}.\] 

Like any recurrence problem, let $r_d = c^d$ for some constant $c$. Then, \begin{center}\begin{align*}c^d &= ac^{d-1} + bc^{d-2} \\ c^2 &= ac + b \\ c^2 - ac - b &= 0 \\ c &= k_1, k_2.\end{align*}\end{center} 

So, we can plug in $c$ to $r_d$.\[r_d = Ak_1^d + Bk_2^d,\] for some constants $A$ and $B$.

Now, we have to use the starting conditions $r_0 = 1, r_1 = \floor{k}$ to find $A$ and $B$. This is simply done by plugging in $d = 0, 1$. However, note that when $b \geq 0$, $\floor{k} = a$, and when $b < 0$, then $\floor{k} = a-1.$ 

So, when $b \geq 0$, we can solve two linear systems.
\[
\begin{bmatrix}
1 & 1 & 1\\
k_1 & k_2 & a
\end{bmatrix} \to \begin{bmatrix}
1 & 0 & \dfrac{a + \sqrt{a^2 + 4b}}{2\sqrt{a^2 + 4b}} = \dfrac{k_1}{\sqrt{a^2+4b}}\\[2em]
0 & 1 & \dfrac{-a + \sqrt{a^2 + 4b}}{2\sqrt{a^2 + 4b}}= \dfrac{-k_2}{\sqrt{a^2+4b}}
\end{bmatrix}
\]
So, if $b \geq 0$, \[r_d = \frac{k_1^{d+1} - k_2^{d+1}}{\sqrt{a^2+4b}}.\] 

By solving the same original matrix with $a-1$ swapped with $a$, we can get that when $b < 0$, \[r_d = \frac{(k_1-1)^{d+1} - (k_2-1)^{d+1}}{\sqrt{a^2+4b}}.\]

Now, we can plug these closed formulas into $\rho(k)$. When $b \geq 0,$ \begin{align*}\rho(k) &= \lim_{d \to \infty} \dfrac{r_d}{k_1^d} \\ &= \lim_{d \to \infty} \dfrac{k_1^{d+1} - k_2^{d+1}}{k_1^d\sqrt{a^2+4b}} \\ &= \lim_{d \to \infty} \dfrac{k_1^{d+1}}{k_1^d\sqrt{a^2+4b}} \\ &= \dfrac{k_1}{\sqrt{a^2+4b}}. \end{align*}

Similarly, when $b < 0$, \[\rho(k) = \dfrac{k_1-1}{\sqrt{a^2+4b}}.\]

Note that when $b = 0$, then $k = a \in \Z$, so we can also write $\rho(k) = \dfrac{a-1}{a}$ due to $r_d = k^{d-1}(k-1).$ 
\end{proof}

The values described in \cref{cor:exact-rho} are shown in \cref{tab:k_values}.

\begin{table}[H]
    \centering
    \includegraphics[width = \textwidth]{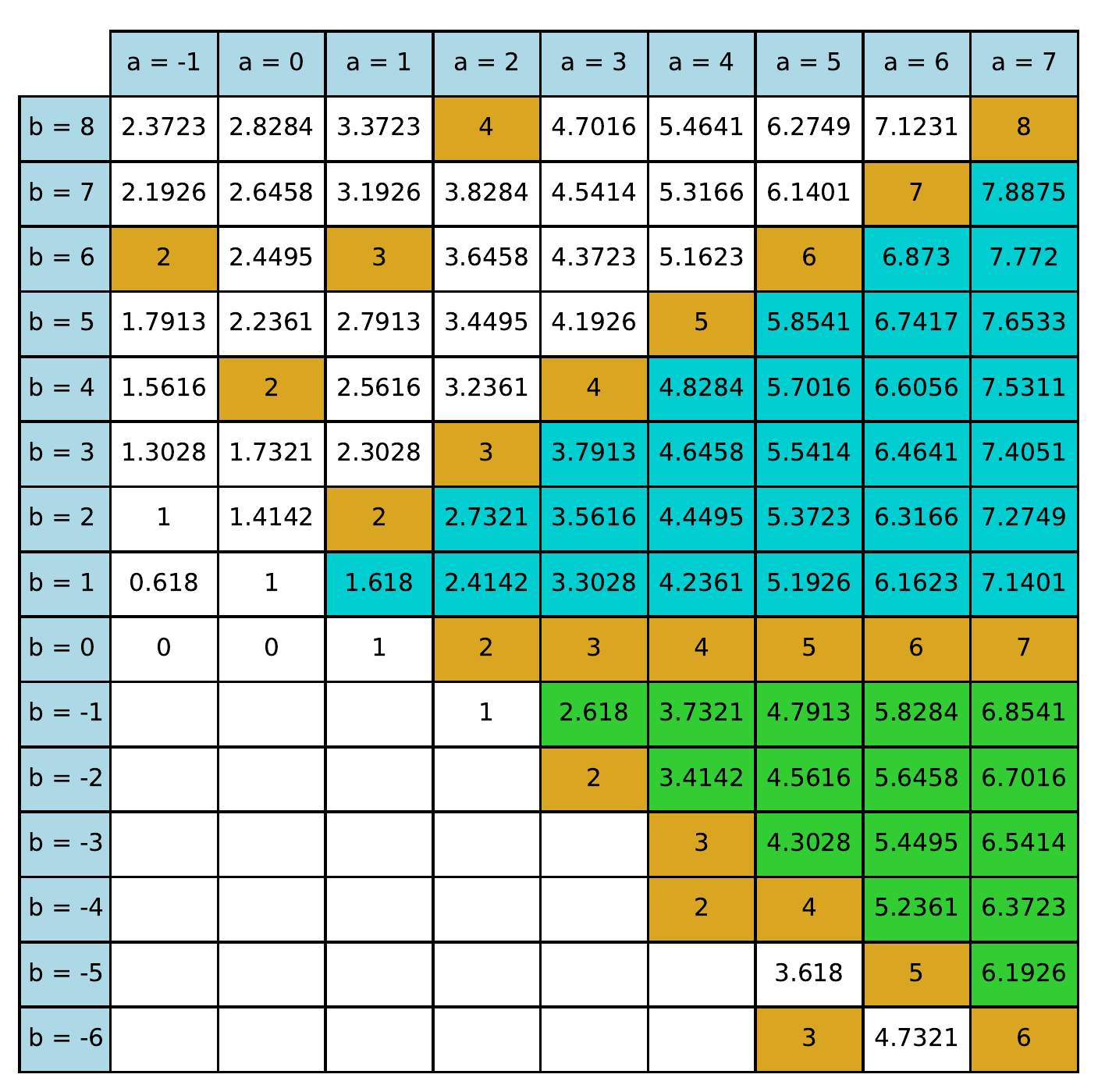}
    \caption{Values of the function $k = \frac{a + \sqrt{a^2 + 4b}}{2}$ for $a,b\in\Z$. Calculated the formula $k = \dfrac{a + \sqrt{a^2 + 4b}}{2}$ with decimal approximations for $-1 \leq a \leq 7,$ and $-6 \leq b \leq 8$.}
    \label{tab:k_values}
\end{table}

So, by using \cref{cor:exact-rho}, we can look at the closed formula for points that we do know on \cref{fig:special_rho}.

\begin{figure}[H]
    \centering
    \includegraphics[width=\textwidth]{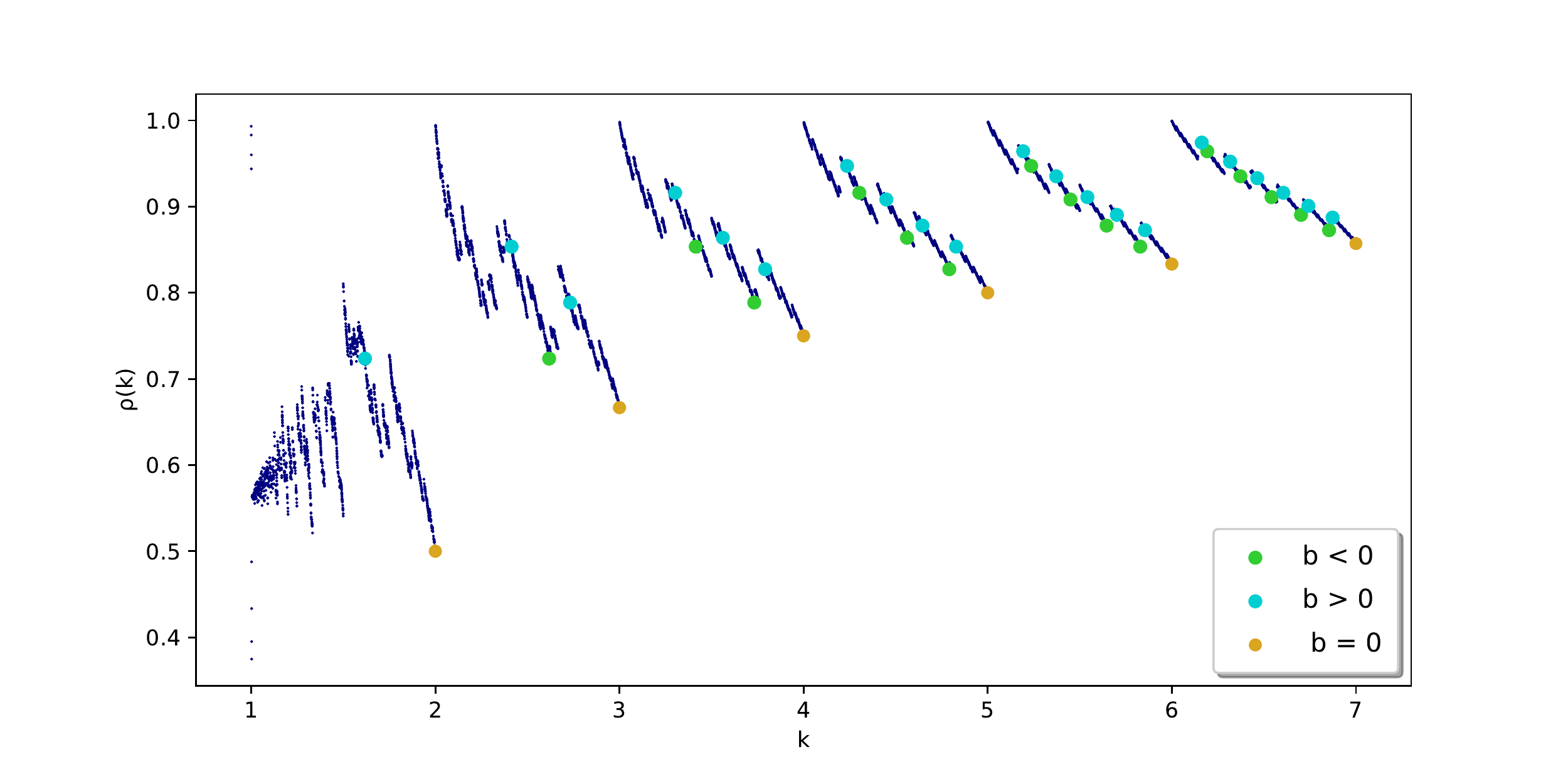}
    \caption[$\rho(k)$ for $1 < k \leq 7$ with closed form]{Approximated $\rho(k)$ for $10^4$ points, and also plotted the points that contain a closed formula.}
    \label{fig:special_rho}
\end{figure}

\section{Conclusion}

So, we have found a closed formula for the irrational case, even when the rational case doesn't have the same type of behavior. This is a rare case where irrationality seems to behave more nicely than their rational counterpart. Also, we have a closed formula for these golden-like trees. This is interesting and motivates the extension from a $k$-ary tree.


Now, for future work, there seems to be three different modes of progress. First would simply be to change from $\floor{\frac{n}{k}}$ to $\ceil{\frac{n}{k}}$. This would be a $k$-ascending tree.

\begin{definition}
\label{def:k-asc-tree}
For any $k \in (1,\infty) \subset \R$, the \emph{$k$-ascending tree} is the rooted tree with nodes in $\N$, where every $n \in \N$ has the parent $\ceil{\frac{n}{k}}$.
\end{definition}

Note that the root node of \cref{def:k-asc-tree} is not going to be $0$. 

\begin{figure}[H]
    \centering
    \includegraphics[width=\textwidth]{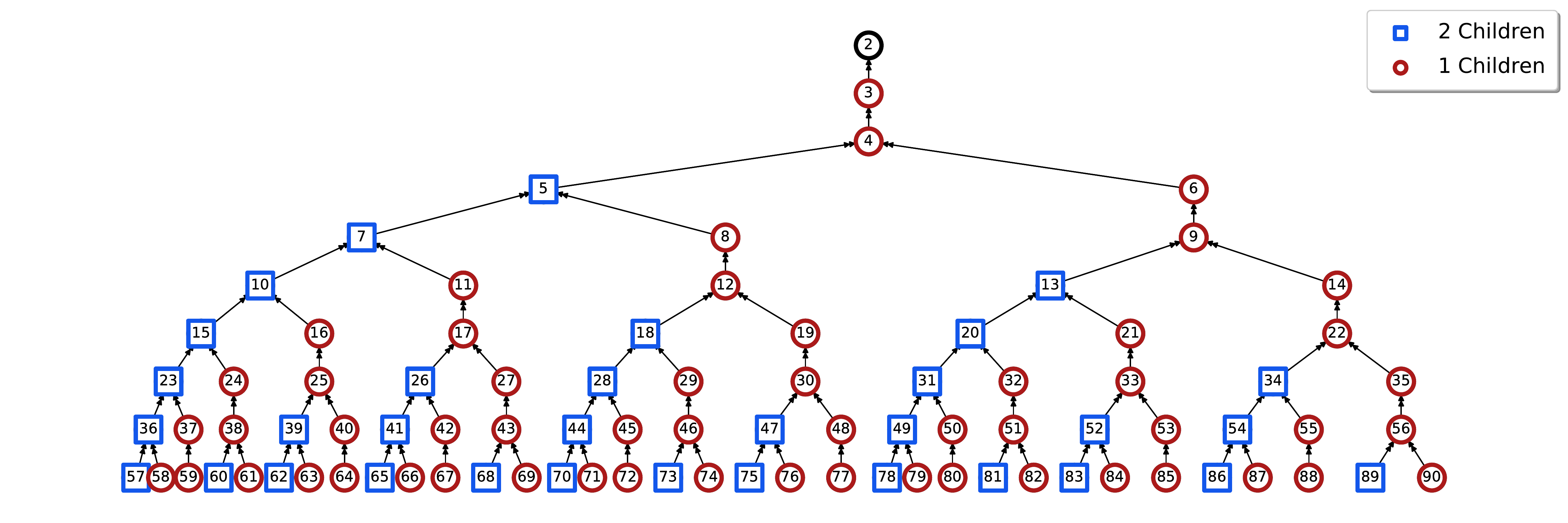}
    \caption[$\phi$-ascending tree]{This generates a $\phi$-ascending tree. Modified the original code to use the ceiling function instead of floor}
    \label{fig:phi-asc-tree}
\end{figure}

\begin{figure}[H]
    \centering
    \includegraphics[width=\textwidth]{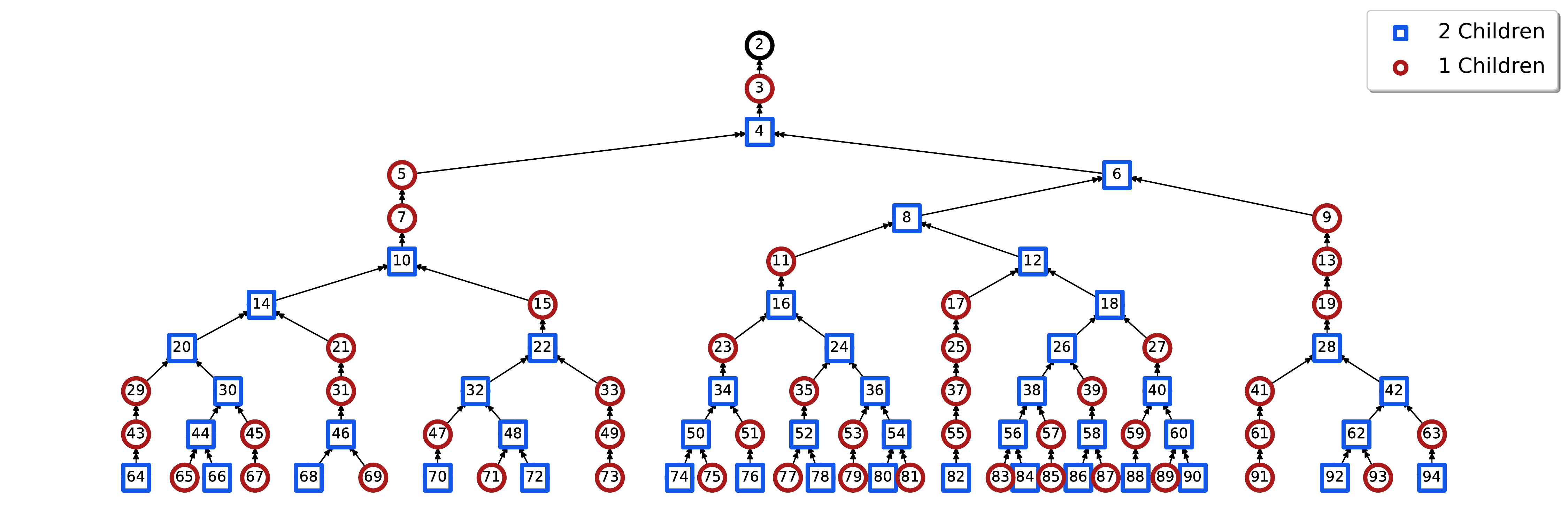}
    \caption[$\frac{3}{2}$-ascending tree]{This generates a $\frac{3}{2}$-ascending tree. Note that this is still a rhythmic tree with a shift in indexing.}
    \label{fig:32-asc-tree}
\end{figure}

From observation, both examples seem to be shifted (as the root is not $0$), and \cref{fig:phi-asc-tree} also seems to be flipped. From this it seems that a lot of the same ideas would carry over very nicely.

Second, it may be interesting to look at more complex $k$'s of the form $k = a + \frac{b}{k} + \frac{c}{k^2}.$ Instead of \cref{thm:recurrence}'s $r_d = a\cdot r_{d-1} + b \cdot r_{d-2}$, we may have the recurrence $r_d = a\cdot r_{d-1} + b\cdot r_{d-2} + c \cdot r_{d-3}.$ It would be nice if this was true, but it doesn't seem like it would work due to \cref{thm:recurrence}'s reliance on the grandparent indicator only going back to $r_{d-2}$.

Also, note that the $\phi$-tree has a regular language that describes its paths. In \cite{marsault2014rhythmic}, they say that all rational numbers $k$ that are not integers do not have this property in the $k$-tree, so it would be interesting to see which other irrational numbers have this property. This isn't true for other ``golden-like" numbers, or numbers of the form $k = a + \frac{b}{k}.$

Furthermore, one of the original reasons that \cite{marsault2014rhythmic} researched rhythms was their connection to fractional bases. When $k$ is a rational number, the language describing paths in the $k$-tree form a bijection with the natural numbers in base $k$. In the irrational $k$ case, the path to a node $n$ in the $k$-tree could be seen as some sort of base-$k$ representation of $n$, but addition is not as clear anymore. 

Within this paper there was the surprising connection to the Josephus problem \cite{odlyzko1991functional}, and the article \cite{grime2022beyond} looks specifically at the constant $\rho(\frac{3}{2})$ in regard to the Collatz Conjecture and mentions the Josephus problem. The connection to the Collatz Conjecture is light, as one can simplify the process into this tree. The Collatz Conjecture states that if you take some $n \in \N$, if it is even divide it by 2, otherwise take $3n+1$ and continue this process, you will always reach 1. The article talks about the simplification of this process to just $\floor{\frac{3n}{2}}$ to remove the complexity and arrives at the $c(\frac{3}{2})$ constant to describe the tree that this process generates, which is the same as our $k$-descending tree. 

These connections to the Josephus Problem and others can be further fleshed out if one found a more direction connection by figuring out how to represent the Josephus problem with trees. One idea could be that the indices of the vertices represent the players in the Josephus problem. 

This research was presented at the Sofia, Bulgaria WFNMC conference in the summer of 2022. The authors would like to thank the administration there.

\bibliographystyle{plainurl}
\bibliography{refs}

\hspace{3cm}~ \break

\href{https://www.linkedin.com/in/agniv-sarkar-88762a242/}{Agniv Sarkar} \and \href{https://eric-severson.netlify.app/}{Eric Severson}
\addresseshere

\clearpage
\pagenumbering{roman}
\appendix

\section{Tree Generation Code}\label{appendix:tree-code}

This is the code used to generate the $k$-tree figures. 

\UseRawInputEncoding
\begin{lstlisting}[language=Python, caption=Tree Generation Code]
from decimal import *
from re import X
import networkx as nx
import matplotlib.pyplot as plt
from math import floor, ceil
import os
import random

save_as = 4
save_location = "./drive/MyDrive/Figures"
largest_can_draw = 1000

os.makedirs(save_location, exist_ok=True)

def rows(n, k):
    """ returns r0...r_n in the k-tree """
    r = row_gen(Decimal(k))
    return [next(r) for x in range(n)]

def left_gen(k):
  """ generator for leftmost nodes in the k-tree """
  yield 0
  l = 1
  while True:
    yield l
    l = ceil(Decimal(l) * Decimal(k))

def row_gen(k):
  """ generator for row lengths in the k-tree"""
  left_iter = left_gen(k)
  last_left = next(left_iter)
  while True:
    next_left = next(left_iter)
    yield next_left - last_left
    last_left = next_left

def fancytree(largest, k, save = True):
    """ 
    Draws a k-tree up to largest value given. 
    To use depth instead, pass on the max_n from here:

    depth = 10 # some number of rows
    max_n = 0
    row_lengths = rows(depth, k)
    for row in row_lengths:
        max_n += row

    """
    G = nx.DiGraph()
    fl = [] 
    ce = [] 

    # CONNECTING THE NODES
    for x in range(1, largest):
        G.add_edge(floor(x/k), x)
        if Decimal(k*x) - floor(Decimal(k*x)) < 1 - k + floor(k):
            fl.append(x)
        else:
            ce.append(x)

    # DRAWING THE TREE
    if (largest > largest_can_draw):
        return
    plt.clf()
    pos = hierarchy_pos(nx.Graph(G), 0)
    G.add_edge(0, 0)

    # DRAWS THE NODES WITH COLOR
    nsize = 450 
    lwidth = 3.5 
    nx.draw(G.reverse(), pos, nodelist = [0], node_size = nsize, node_color = '#ffffff', node_shape = "o", edgecolors='#000000', linewidths = lwidth, with_labels = True)
    nx.draw(G.reverse(), pos, label = str(floor(k)) + " Children", nodelist = fl, node_size = nsize, node_color = '#ffffff', node_shape = "s", edgecolors='#1357eb', linewidths = lwidth, with_labels = True)
    nx.draw(G.reverse(), pos, label = str(ceil(k)) + " Children", nodelist = ce, node_size = nsize, node_color = '#ffffff', node_shape = "o", edgecolors='#ab1a1a', linewidths = lwidth, with_labels = True)
    if abs(k - int(k)) <= 10**(-5):
      pass
    else:
      plt.legend(loc = 'best', fontsize = 'xx-large', markerscale = .45, 
                labelspacing = .7,
                title_fontsize = 'xx-large', borderpad = .6, shadow = True)
    if save:
      plt.savefig(save_location+str(round(k, save_as))+"tree.pdf", format = 'pdf', dpi=300)
    plt.show()

def hierarchy_pos(G, root=None, width=1., vert_gap = 0.2, vert_loc = 0, leaf_vs_root_factor = 0.5):
    """
    Source: https://github.com/ryoji-kubo/CurvatureAnalysis/blob/main/utils.py
    Source: https://github.com/BaseMax/BinaryTreeDiagramDrawing/blob/master/tree.py
    """

    if root is None:
        if isinstance(G, nx.DiGraph):
            root = next(iter(nx.topological_sort(G)))  #allows back compatibility with nx version 1.11
        else:
            root = random.choice(list(G.nodes))
      
    if not nx.is_tree(G):
        raise TypeError('cannot use hierarchy_pos on a graph that is not a tree')

    def _hierarchy_pos(G, root, leftmost, width, leafdx = 0.2, vert_gap = 0.2, vert_loc = 0, 
                    xcenter = 0.5, rootpos = None, 
                    leafpos = None, parent = None):

        if rootpos is None:
            rootpos = {root:(xcenter,vert_loc)}
        else:
            rootpos[root] = (xcenter, vert_loc)
        if leafpos is None:
            leafpos = {}
        children = list(G.neighbors(root))
        leaf_count = 0
        if not isinstance(G, nx.DiGraph) and parent is not None:
            children.remove(parent)  
        if len(children)!=0:
            rootdx = width/len(children)
            nextx = xcenter - width/2 - rootdx/2
            for child in children:
                nextx += rootdx
                rootpos, leafpos, newleaves = _hierarchy_pos(G,child, leftmost+leaf_count*leafdx, 
                                    width=rootdx, leafdx=leafdx,
                                    vert_gap = vert_gap, vert_loc = vert_loc-vert_gap, 
                                    xcenter=nextx, rootpos=rootpos, leafpos=leafpos, parent = root)
                leaf_count += newleaves

            leftmostchild = min((x for x,y in [leafpos[child] for child in children]))
            rightmostchild = max((x for x,y in [leafpos[child] for child in children]))
            leafpos[root] = ((leftmostchild+rightmostchild)/2, vert_loc)
        else:
            leaf_count = 1
            leafpos[root]  = (leftmost, vert_loc)
        return rootpos, leafpos, leaf_count

    xcenter = width/2.
    if isinstance(G, nx.DiGraph):
        leafcount = len([node for node in nx.descendants(G, root) if G.out_degree(node)==0])
    elif isinstance(G, nx.Graph):
        leafcount = len([node for node in nx.node_connected_component(G, root) if G.degree(node)==1 and node != root])
    rootpos, leafpos, leaf_count = _hierarchy_pos(G, root, 0, width, 
                                                    leafdx=width*1./leafcount, 
                                                    vert_gap=vert_gap, 
                                                    vert_loc = vert_loc, 
                                                    xcenter = xcenter)
    pos = {}
    for node in rootpos:
        pos[node] = (leaf_vs_root_factor*leafpos[node][0] + (1-leaf_vs_root_factor)*rootpos[node][0], leafpos[node][1]) 
    xmax = max(x for x,y in pos.values())
    for node in pos:
        pos[node]= (pos[node][0]*width/xmax, pos[node][1])
    return pos
\end{lstlisting}

\section{\texorpdfstring{$\rho$}{\textit{rho}} Code}\label{appendix:rho-code}

This is the code used to generate $\rho(k)$ values. 

\UseRawInputEncoding
\begin{lstlisting}[language=Python, caption=$\rho(k)$ Approx. Code]
import pandas as pd
import numpy as np
from math import floor, ceil, sqrt, log
from decimal import *
import matplotlib.pyplot as plt

source = "./drive/MyDrive/Figures/rho.csv"

def resetRho():
  """ Empties out if error is encountered """
  df = pd.DataFrame({'val':[],'out':[],'precision':[]})
  df.to_csv(source, index=False)
  return df
  
def getPrec(k, df):
  """ Returns the stored precision for rho(k) in the csv. 0 if not found """
  try:
    return df.loc[min(df.loc[df.val == k].index)].precision
  except:
    return 0

def addRho(df, info):
  """ Adds value to the stored csv. """
  inside = getPrec(info[0], df) != 0
  if inside:
    df.loc[(df["val"] == info[0])] = info
  else:
    df.loc[len(df)] = info
  return df
 
def writeRho(df):
  """ Writes to an external csv for storage """
  df.to_csv(source, index = False)

def left_gen(k):
  """ Generator for leftmost nodes """
  yield 0
  l = 1
  while True:
    yield l
    l = ceil(Decimal(l) * Decimal(k)) 

def row_gen(k):
  """ Generator for row lengths """
  left_iter = left_gen(k)
  last_left = next(left_iter)
  while True:
    next_left = next(left_iter)
    yield next_left - last_left
    last_left = next_left

def get_rho(k, min_row_size = 10**3, eps = 10 ** -3, iter_upper_bound = 10 ** 2):
  """
  If a rho(k) value is not found within rho.csv, generate it. 
  Also runs if a higher precision value 
  precision is defined with:
  minimum row size to start at, upper bound to stop at, eps, decimal precision
  """
  if k <= 1:
    return None
  elif abs(k - int(k)) < 10**(-1000000):
    return Decimal((k-1))/Decimal(k)
  row = row_gen(k)
  last_c = 1
  r = next(row)
  c = 1
  while r < min_row_size:
    r = next(row)
    c += 1
  for i in range(c, iter_upper_bound+c):
    next_c = Decimal(next(row)) / Decimal(k) ** Decimal(i)
    if abs(next_c - last_c) < eps:
      return (next_c + last_c) / 2
    last_c = next_c
  return next_c

def find_rho(k, min_row_size = 10**3, eps = 10 ** -3, iter_upper_bound = 10 ** 2):
  """
  Gets the value of rho(k) to some precision, either through rho.csv or generating it.
  """
  try:
    df = pd.read_csv(source)
  except FileNotFoundError:
    return get_rho(k, min_row_size=min_row_size, eps=eps, iter_upper_bound=iter_upper_bound)
  except:
    df = resetRho()
  if abs(k - int(k)) < 10**(-100):
    prec = np.inf
  else:
    prec = log(min_row_size*eps**(-1)*iter_upper_bound*getcontext().prec, 10) 
  curPrec = getPrec(k, df)
  if curPrec < prec or curPrec == 0:
    rho = get_rho(k, min_row_size=min_row_size, eps=eps, iter_upper_bound=iter_upper_bound)
    writeRho(addRho(df, [k, rho, prec]))
    return rho
  elif curPrec >= prec:
    return df.loc[min(df.loc[df.val == k].index)].out

x = np.linspace(1 + 10**(-6), 6, 10**5)
plt.scatter(x, np.array(list(map(find_rho, x))), s = 0.1)
plt.xlabel('k')
plt.ylabel('ρ(k)')
plt.title(f'Estimated values of ρ(k) for 100000 random values 1 < k < 6')
plt.show()
\end{lstlisting}

\section{\texorpdfstring{$\rho$}{\textit{rho}} Raw Data}\label{appendix:raw}

This is some recorded raw data generated by the code. Note that there appears to be some rounding off that happens as the data is written to the csv file. The data is actually too large to be typically viewed with tools like csvsimple. The link is at \href{https://github.com/agniv-the-marker/rho-stuff}{https://github.com/agniv-the-marker/rho-stuff}.

\section{Grandparent Indicator Figures}\label{appendex:grand-figures}

\begin{figure}[H]
    \centering
    \includegraphics[width=\textwidth]{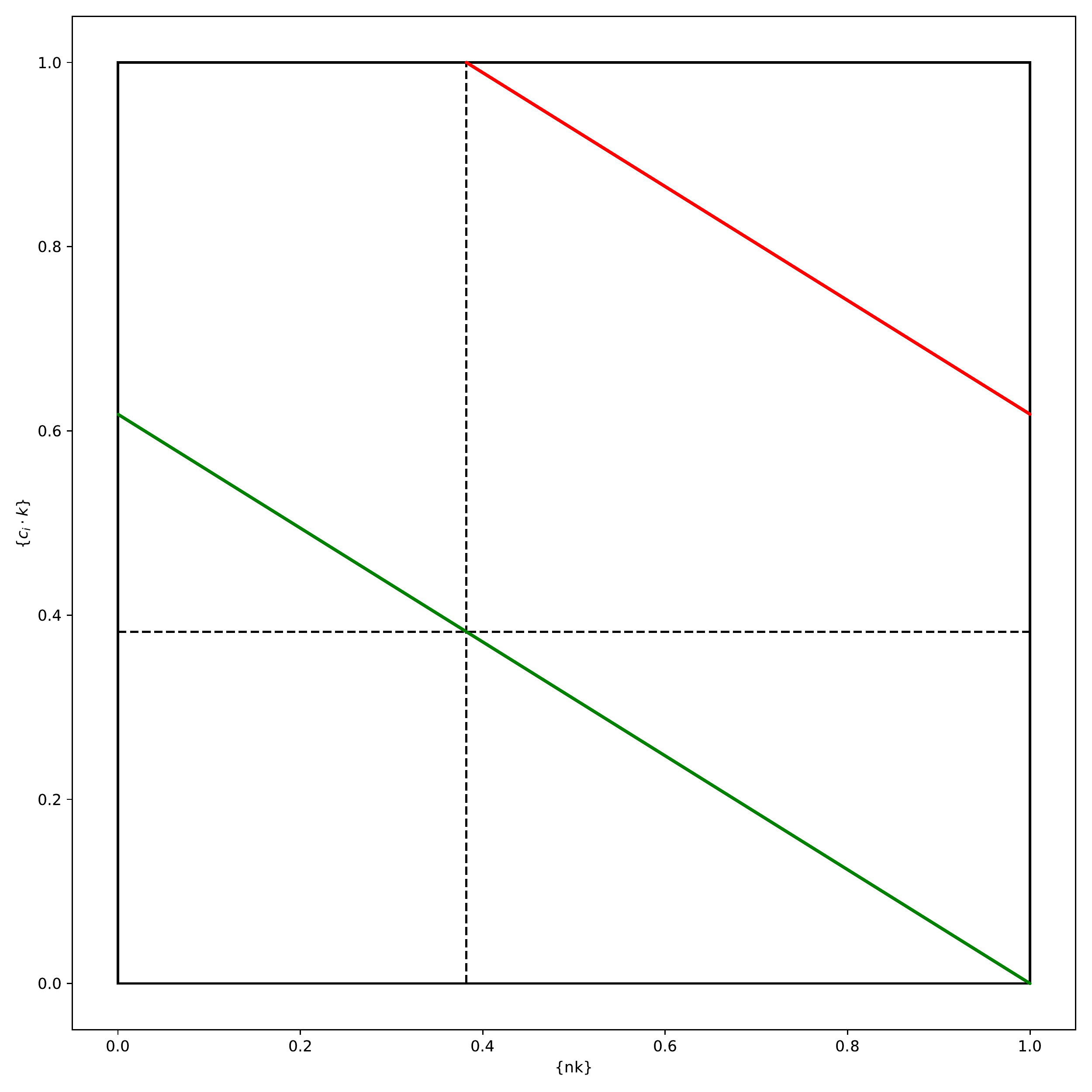}
    \caption{$\phi$ Grandparent Indicators}
    \label{fig:phi-grand}
\end{figure}
\begin{figure}[H]
    \centering
    \includegraphics[width=\textwidth]{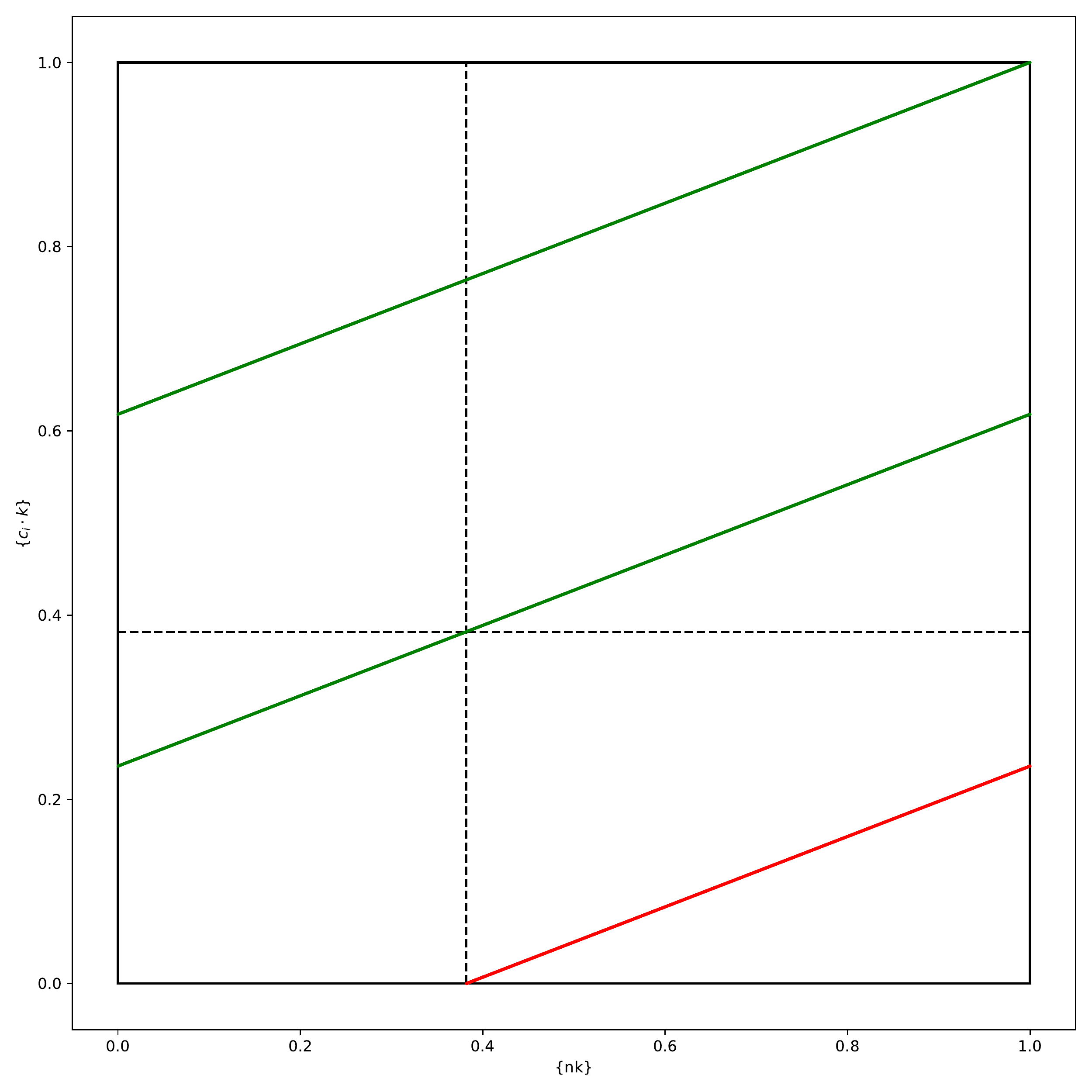}
    \caption{$a = 3, b = -1$ Grandparent Indicators}
    \label{fig:3-1-grand}
\end{figure}
\begin{figure}[H]
    \centering
    \includegraphics[width=\textwidth]{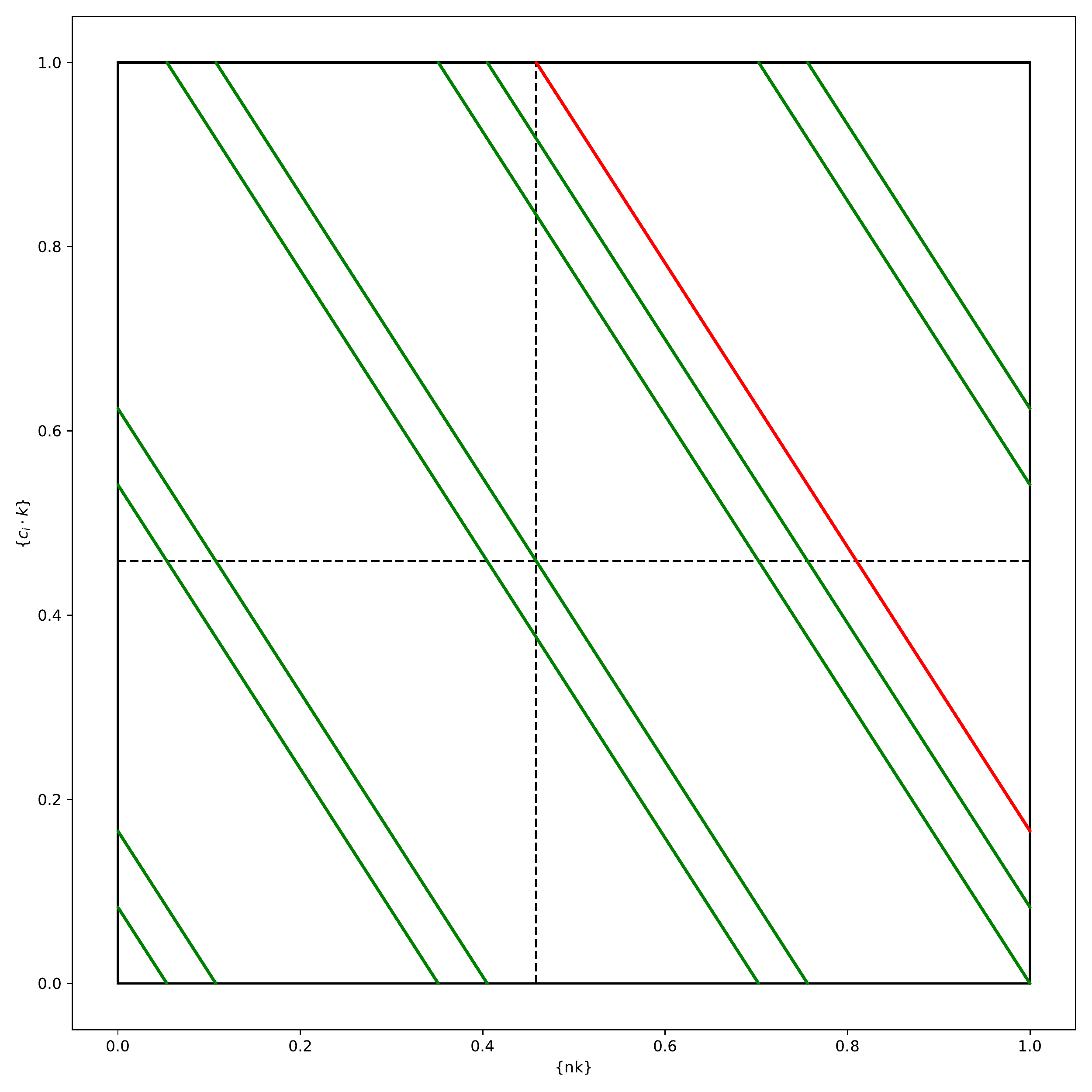}
    \caption{$a = 3, b = 7$ Grandparent Indicators}
    \label{fig:37-grand}
\end{figure}
\begin{figure}[H]
    \centering
    \includegraphics[width=\textwidth]{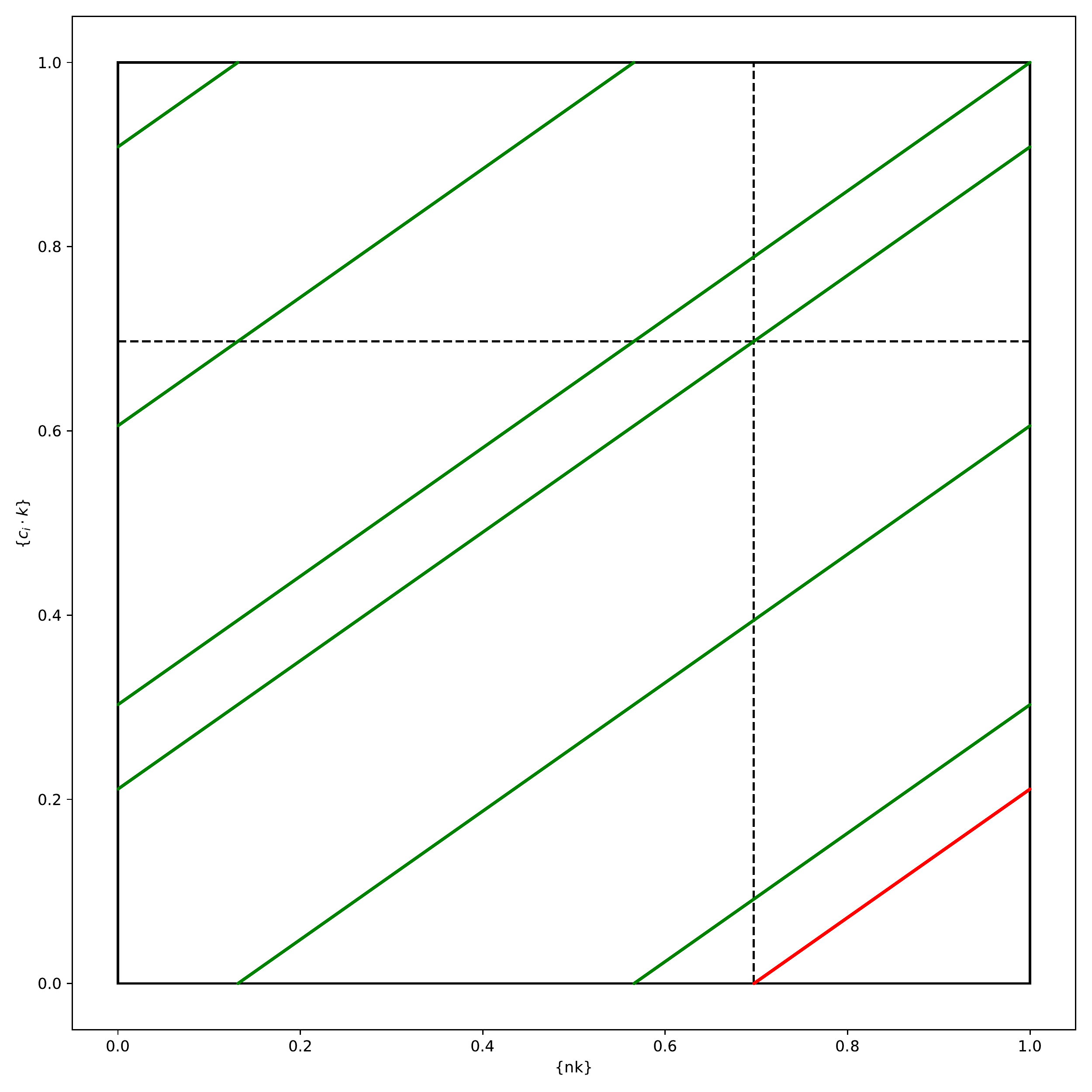}
    \caption{$a = 5, b = -3$ Grandparent Indicators}
    \label{fig:5-3-grand}
\end{figure}
\begin{figure}[H]
    \centering
    \includegraphics[width=\textwidth]{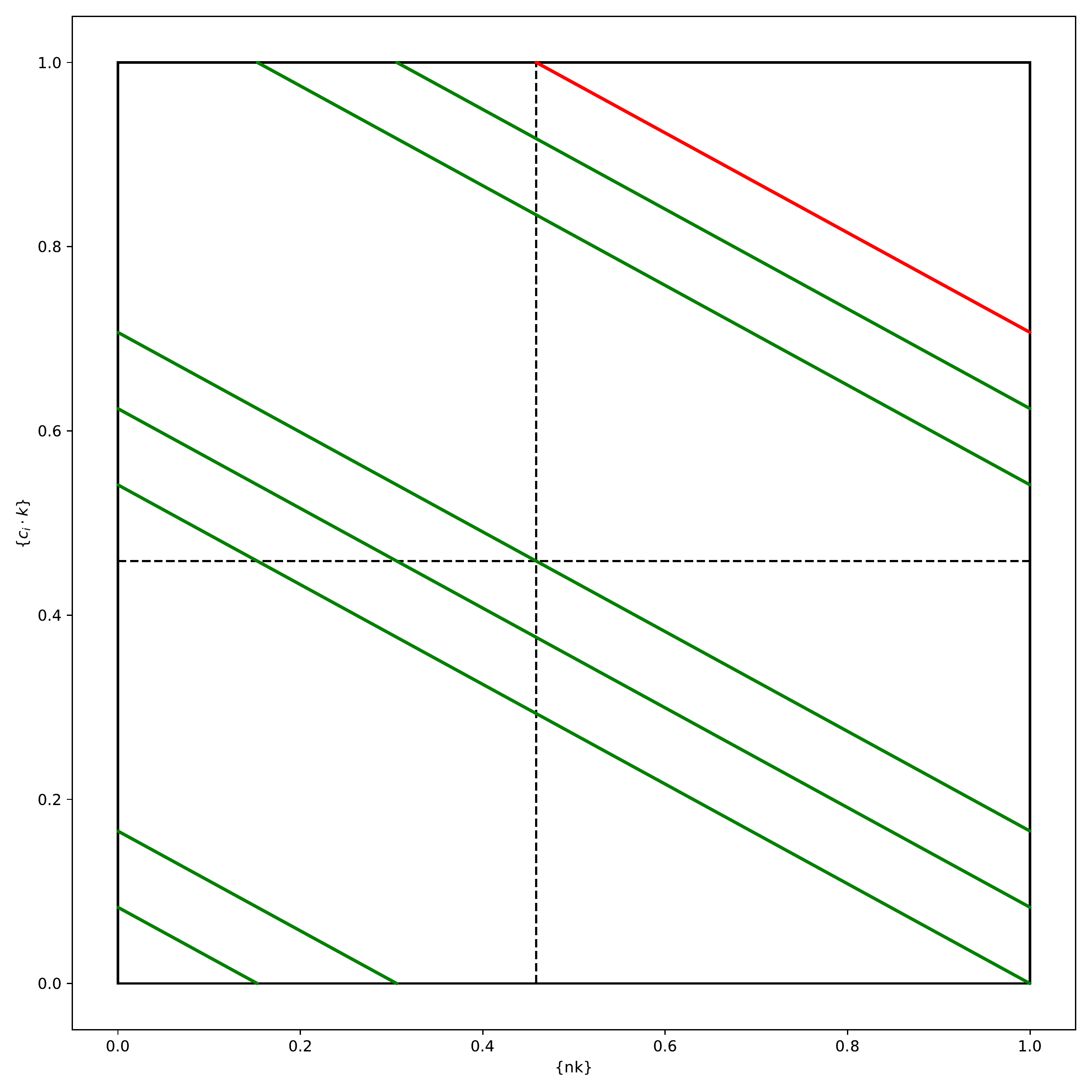}
    \caption{$a = 5, b = 3$ Grandparent Indicators}
    \label{fig:53-grand}
\end{figure}


\end{document}